\documentclass[11pts]{article}

\usepackage[margin=1in]{geometry}

\usepackage{amsmath,amssymb,latexsym,amsthm,graphicx}

\usepackage{mathrsfs}

\usepackage{url}
\usepackage{wrapfig}
\usepackage{subfigure}
\usepackage{hyperref}
\usepackage[final]{pdfpages}
\usepackage{fancyhdr}
\usepackage{caption}
\newtheorem{theorem}{Theorem}[section]
\newtheorem{lemma}[theorem]{Lemma}

\newtheorem{coro}[theorem]{Corrolary}
\newtheorem{defi}[theorem]{Definition}
\newtheorem{prob}[theorem]{Problem}
\newtheorem{ass}[theorem]{Assumption}

\newcommand{\Id}{\mathbb{I}}

\newcommand{\mE}{\mathcal{E}}

\newcommand{\rN}{\mathbb{R}}

\newcommand{\mS}{{\mathcal S}}

\newcommand{\vv}{\mbox{\bf v}}

\newcommand{\ww}{\mbox{\bf w}}

\newcommand{\eg}{\varepsilon}

\newcommand{\eps}{\epsilon}

\newcommand{\llg}{\lambda}

\newcommand{\Og}{\Omega}

\newcommand{\pdh}{\partial}

\newcommand{\supp}{\mbox{supp}}

\newcommand{\diam}{\mbox{diam}}

\newcommand{\dist}{\mbox{dist}}

\newcommand{\R}{\mathbb R}

\newcommand{\edot}{\,\cdot\,}

\newcommand{\x}{\mathrm{x}}

\newcommand{\f}{{\bf f}}
\newcommand{\g}{{\bf g}}
\newcommand{\uu}{{\bf u}}
\newcommand{\U}{{\bf U}}

\newcommand{\h}{{\bf h}}

\def \p { \partial}

\vspace{-2cm}
\title{Time reversal method for thermoacoustic acoustic tomography in elastic media: convergence with sharp observation time}
\author{Vitaly Katsnelson\thanks{Department of Computational and Applied Mathematics, Rice University, Houston TX, USA. \texttt{vitaly.katsnelson@rice.edu}}\and
Linh V. Nguyen \thanks{Department of Mathematics, University of Idaho, Moscow, ID 83844, USA. \texttt{lnguyen@uidaho.edu}}}

\date{}
\begin{document}
\maketitle
\begin{abstract}
In this article, we consider the inverse source problem arising in photoacoustic tomography in elastic media. We show that the time reversal method, proposed by Tittelfitz [Inverse Problems 28.5 (2012): 055004], converges with the sharp observation time without any constraint on the speeds of the longitudinal and shear waves.
\end{abstract}

\section{Introduction} Let us consider the isotropic elastic wave propagation in the free space
\begin{eqnarray} \label{E:TAT} \left \{\begin{array}{l} \uu_{tt}(t,\x) - \Delta^* \uu(t,\x) =0,~ \x \in \R^3,~ t \geq 0,\\[4 pt] \uu(0,\x) = \f(\x),~ \uu_t(0,\x) = 0,~\x \in \R^3, \end{array} \right.\end{eqnarray}
where $\uu=(u_1,u_2,u_3)$ is the displacement vector. Here,
$$\Delta^* \uu  = \nabla \Big[\mu(\x) \big( \nabla \uu + (\nabla \uu)^T \big) \Big] + \nabla ( \llg(\x) \, \nabla \cdot \uu),$$ where $\lambda,\mu$ are Lam\'e parameters, and
$$(\nabla \uu)_{i,j} = \frac{\partial \uu_i}{\partial \x_j}$$ is the Jacobian of $\uu$ and $(\nabla \uu)^T$ is its transpose. We assume that $\lambda=\lambda(\x)$ and $\mu=\mu(\x)$ are positive and bounded. Moreover, $\mu$ is from below by a positive constant.

Let $\Og$ be a bounded domain in $\R^3$ with the boundary $\mS= \partial \Og$. In this article, assuming that $\supp(\f) \subset \Og_0 \Subset \Og$, we are interested in the inverse source problem.
\begin{prob} \label{P:Main}
Find the initial displacement $\f$ given the data $\g = \uu|_{[0,T] \times \mS}$ for some (observation time) $T>0$.
\end{prob}
This problem arises in thermo/photo-acoustic tomography in elastic media. The same problem in the acoustic setting is very well-studied (see, e.g., \cite{FPR,kuchment2008mathematics,HKN,US,stefanov2011thermo,qian2011new,kuchment2014radon}). Problem~\ref{P:Main} was first studied in \cite{tittelfitz12}. In that article, following the work of Stefanov and Uhlmann \cite{US}, the author proposed a time reversal method to solve Problem~\ref{P:Main}. However, in order to prove the convergence of the method, the author assumed that the supremum of  the P-wave speed is less than three times the infimum of the S-wave speed. Moreover, the required measurement time $T$ has to be sufficiently large. In this article, we show that the same algorithm works without the restriction on the wave speeds. The needed observation time $T$ is the sharp observation time which comes from the visibility condition (see Assumption~\ref{A:sharp}).

\section{Notation and statement of the main result}
Let us first introduction some notations.


\begin{defi} Let $U \subset \R^3$ be an open set and $\f,\g:U \to \R^3$. We define the following symmetric bilinear form
$$(\f,\g)_{H(U)} = \int_{U} \lambda(\x) \, (\nabla \cdot \f) (\nabla \cdot \g)  + \frac{\mu(\x)}{2} \, \big[\nabla \f + (\nabla \f)^T \big] \cdot \big[\nabla \g + (\nabla \g)^T \big] \ d \x$$ and the semi-norm
$$\|\f\|_{H(U)} =(\f,\f)_{H(U)} ^{1/2}. $$
\end{defi}

Consider the case $U=\Og_0$. For all $\f \in C_0^\infty(\Og_0)$
$$(\f,\f)_{H(\Og_0)} \geq  \inf_{\x} \mu(\x) \,  \int_{\Og_0} \|\nabla \f + (\nabla \f)^T\|^2 \ d \x \geq C \, \|\f\|_{H^1(\Og_0)}.$$
Here, the second inequality comes from Korn's inequality (see \cite[page 322]{marsden1994mathematical}). Therefore, $\|\edot\|_{H(\Og_0)}$ is a norm in $C_0^\infty(\Og_0)$. Let us denote by  $H_0(\Og_0)$ the completion of $C_0^\infty(\Og_0)$ under that norm. Then, $H_{0}(\Og_0) \cong H^1_0(\Og_0)$.


The following result describes the connection between $(\edot, \edot)_{H(U)}$ and the elastic operator $\Delta^*$ arising in problem (\ref{E:TAT}).
\begin{lemma}\label{L:Int}
For all $\f \in H^1(U)$ and $\g \in H^1_0(U)$, we have
$$(-\Delta^* \f, \g)_{L^2(U)} = (\f,\g)_{H(U)}.$$
\end{lemma}

The proof of this lemma follows from a simple integration by parts argument. We skip it for the sake of brevity.

Let us recall that for each $\f \in H_0(\Og_0) \cong H_0^1(\Og_0)$, problem (\ref{E:TAT}) has a unique solution (see, e.g., \cite{marsden1994mathematical})
$$\uu \in C([0,T] \times H(\rN^3)) \cap C^1([0,T] \times L^2 (\rN^3)).$$
For the notational convenience we will denote $\U=(\uu,\uu_t)$. Then $\U(t) \in H(\R^3) \times L^2(\R^3)$ for all $t \geq 0$.

To motivate the main assumption of this article, let us describe the propagation of singularities of the solution $\uu$ to equation (\ref{E:TAT}). It is well-known that $\uu$ can be decomposed microlocally into two modes: the P-wave  and S-wave (see, e.g., \cite{dencker1982propagation,taylor1978propagation,brytik2011decoupling}). The P-wave propagates with speed $c_p(\x) := \sqrt{\lambda(\x) + 2 \mu(\x)}$ while the S-wave with speed $c_s(\x) := \sqrt{\mu(\x)}$. Let $(\x,\xi) \in S^* \Og$ \footnote{$S^* \Og$ denotes the unit cotangent bundle of $\Og$.} be a singularity of $\f$. It induces either a pair of P-wave singularities propagating on the bicharacteristics rays $\{(t,\tau=  \pm |\xi|_{p},\gamma^p_{\x,\xi}(\pm t),\dot \gamma^p_{\x,\xi}(\pm t)): t \geq 0\}$ or a pair of S-wave singularities propagating on the bicharacteristics rays $\{(t,\tau= \pm |\xi|_{s},\gamma^s_{\x,\xi}( \pm t),\dot \gamma^s_{\x,\xi}(\pm t)): t\geq 0\}$ (or possibly both of them). Here, $\gamma^{p/s}_{\x,\xi}$ is the the unit speed geodesics originating from $\x$ along direction $\xi$ under the metrics
  $d_{p/s}= c^{-2}_{p/s} \, d\x^2$. Let us denote by $\tau^{p/s}_{+} (\x,\xi)>0$ and  $\tau^{p/s}_{-} (\x,\xi)<0$ the times the geodesics $\gamma^{p/s}_{\x,\xi}$ hits the surface $\mS$ (i.e., the time $t = \pm \tau^{p/s}_\pm(\x,\xi)>0$ is the moment the corresponding propagating singularity is observed on the surface $\mS$). We assume that for $t>\tau_+^{p/s}(\x,\xi)$ and $t<\tau^{p/s}_{-}(\x,\xi)$, the geodesics $\gamma^{p/s}_{\x, \xi}(t)$ leaves and does not return to the domain $\Og$. It is understood that (see \cite{tittelfitz2015inverse}) in order for Problem~\ref{P:Main} to be stable, at least one of the aforementioned propagating singularities has to be observed on $\mS$. Therefore, we make the following assumption:

\begin{ass} \label{A:sharp}
For all  $(\x,\xi) \in S^* \Og$, $\min \{\tau_+^{p/s} (\x,  \xi), - \tau_-^{p/s} (\x, \xi) \} < T$.
\end{ass}
In this article, we will prove that the Neumann series method proposed in \cite{tittelfitz12} converges under this sharp condition.  For our convenience, we denote
$$c_+=\sup_\x c_p(\x), \quad c_-=  \inf_\x c_s(\x),$$
and $\ell_{p/s}(\Og)$ the length of the longest geodesics segment inside $\overline \Og$ with respect to the corresponding metrics. We note that the above assumption means $T> \ell_s(\Og)/2$.

Let $U$ be an open subset of $\R^3$ (for latter purposes, $U$ is either $\Og$ or $\Og_0$). We define the elastic extension $\mE_\Og(\h)$ of a function $\h: \partial U \to \R^3$ to be the solution $\phi$ of the elliptic problem
\begin{eqnarray*}
\Delta^* \phi =0,~\x \in U,\quad \phi|_{\partial U} = \h.
\end{eqnarray*}
We also define the projection operator $\mathcal{P}_U :H^1(U) \to H_0(U)$ by $\mathcal{P}_U(\f) = \f - \mE_U(\f|_{\p U})$.

\begin{lemma} \label{L:proj} For any $\f \in H^1(U)$:
$$\|\mathcal{P}_U (\f)\|_{H(U)} \leq \|\f\|_{H(U)}. $$
\end{lemma}

\begin{proof} We have
\begin{eqnarray*} \|\mathcal{P}_U(\f) \|^2_{H(U)} &=& \left(\f - \mE_U(\f),\f - \mE_U(\f)  \right)_{H(U)} .\end{eqnarray*}
Applying Lemma~\ref{L:Int} and noting that $\Delta^* \mE_U(\f) =0$, we obtain
\begin{eqnarray*} \|\mathcal{P}_U(\f) \|^2_{H(U)} &=& \left(-\Delta^* (\f - \mE_U(\f)), \f - \mE_U(\f)  \right)_{L^2(U)} = \left(-\Delta^* (\f + \mE_U(\f)), \f - \mE_U(\f)  \right)_{L^2(U)} \\ &=& \left(\f + \mE_U(\f) , \f - \mE_U(\f)\right)_{H(U)} = \|\f\|^2_{H(U)} -  \|\mE_U(\f) \|^2_{H(U)} \leq \|\f\|_{H(U)}. \end{eqnarray*}

\end{proof}

We are now ready to describe the time reversal method (or Neumann series solution) for solving Problem~\ref{P:Main}.  Let $\phi = \mE(\g(T))$ and consider the time-reversal problem:
\begin{eqnarray*}  \left \{\begin{array}{l} \vv_{tt}(t,\x) - \Delta^* \vv(t,\x) =0,~ (t,\x)  \in (0,T) \times  \Og,\\ \vv|_{(0,T) \times \Gamma} = \g, \\[4 pt] \vv(T,\x) =\phi(\x),~ \uu_t(T,\x) = 0,~\x \in \Og, \end{array} \right.\end{eqnarray*}
This problem has a solution $\vv \in C(0,T;H^1(\Og)) \cap C^1(0,T;L^2(\Og))$ (see \cite[Lemma~A.1]{lin2008unique}).
Let us denote $A \g = \mathcal{P}_{\Og_0}(\vv(0)).$ The following theorem gives us a Neumann series to invert the mapping $\Lambda: \f \to \g$.

\begin{theorem}  \label{T:Main} Under Assumption~\ref{A:sharp}, the operator $K=\Id - A \Lambda$ is a contraction from $H_0(\Og_0)$ into itself.
Consequently, the function $\f$ can obtained from the Neumann series
$$\f= \sum_{j=0}^\infty K^j A \g.$$
\end{theorem}
Let us mention that the same result was obtained in \cite{tittelfitz12} under the assumption that $\frac{1}{3} c_+ <c_-$ and $T$ has to satisfy a much stronger condition than Assumption~\ref{A:sharp}. \footnote{Roughly speaking, the condition looks like $T > \frac{\diam(\Og)}{3 c_- - c_+}.$} Our proof follows closely that in \cite{tittelfitz12}. To avoid the aforementioned assumption, we use a simple geometric argument (distance peeling), presented in Lemma~\ref{L:Injective}. We also has to derive an observability estimate for the elastic wave equation (\ref{E:TAT}) in Proposition~\ref{L:lb}. 

\section{Proof of Theorem~\ref{T:Main}}

%
Let us start by recalling two essential results on the unique continuation and domain of dependence of the wave equation. The presented form was formulated in \cite{tittelfitz2015inverse}.
\begin{theorem}[Unique continuation principle] \label{T:Unique}
Suppose that $\lambda,\mu \in C^3(\R^3)$  and that for all $t \in (t_0,t_1)$, $\uu(t)=0$ in a neighborhood of $\{\x_0\}$. Then, $\uu=0$ for all $(t,\x)$ inside the double cone  $$\left \{(t,\x): \dist_{s}(\x,\x_0) + \left| t- \frac{t_0+t_1}{2} \right|  \leq \frac{t_1-t_0}{2} \right \}.$$
\end{theorem}

\begin{theorem}[Domain of dependence principle] \label{T:Domain}
Assume that $\lambda,\mu \in C^1(\R^3)$. Then, $u$ has finite speed of propagation speed with maximum speed $c=c^+$. That is, for any $t_0 \in [0,T)$, if $\uu(t_0, \edot) = \uu_t(t_0,\edot)=0$ in $B_{c(t-t_0)}(\x_0)$ then $\uu=0$ in
$$\left \{(t,\x): t_0 \leq t \leq t_1, ~ \x \in B_{c(t_1-t)}(\x_0) \right\}.$$
\end{theorem}
Here is a simple corollary of Theorem~\ref{T:Domain} which will be used later in this article:
\begin{coro}\label{C:Domain}
Assume that $\lambda,\mu \in C^1(\R^3)$. For any $t_0 \in [0,T)$, if $(\uu(t_0), \uu_t(t_0)) \in H^1(B_{c(t-t_0)}(\x_0)) \times L^2(B_{c(t-t_0)}(\x_0)) $ then $\uu \in H^1$ in
$$\left \{(t,\x): t_0 \leq t \leq t_1, ~ \x \in B_{c(t_1-t)}(\x_0) \right\}.$$
\end{coro}

The following two lemmas contain our main ingredients for the proof of Theorem~\ref{T:Main}.

\begin{lemma} \label{L:lb} Under the visibility condition in the form of Assumption \ref{A:sharp}, 
$$\|\f\|_{H^1(\Og)} \lesssim \|\uu(T)\|_{H^1(\Og^c)} + \|\uu_t(T)\|_{L^2(\Og^c)}+  \|\uu\|_{L^2((0,T)\times \R^3)}. $$
\end{lemma}
There are several possible approaches for the proof. Our proof below share some similarity with that of \cite[Lemma 7]{chervova2016time}.
\begin{proof}
First, since $\uu_t(0) = 0$, one may do an even extension of $\uu$ to $(-T,T) \times \R^3$. Then, $\uu \in C(\R, H^1(\R^3)) \cap C^1(\R, L^2(\R^3))$ satisfies the wave equation in $\R$ for all $t$. In particular $\U(-T)$ also belongs to $H^1(\Omega^c)\times L^2(\Omega^c)$ as well.

In order to prove Lemma~\ref{L:lb}, we will make use of the closed graph theorem. To this end, let us define
\begin{align*}
\mathcal{X}:= \{\uu \in L^2((-T,T) \times \R^n): \Box^* \uu = 0,~ \U(\pm T)|_{\Omega^c} \in H^1(\Omega^c) \times L^2(\Omega^c) \},
\end{align*} where $\Box^*= \p_{tt} -\Delta^*$ is the elastic wave operator.
Initially, one merely has the inclusion map $\iota: \mathcal{X} \to  L^{2}((-T,T)\times \R^n)$. However, we will prove the the following claim:

\[
\iota: \mathcal{X} \to \mathcal{Y}:= \{
\uu \in L^2((-T,T) \times \R^n): \Box^* \uu = 0,~(\uu(0),\uu_t(0)) \in H^1(\Og_0)\times L^2(\Og_0)
\}.
\]
Once the claim is proved, one can easily see that $\iota: \mathcal{X} \to \mathcal{Y}$ has closed graph (since their norms are stronger than the $L^2((-T,T) \times \R^n)$ norm). By applying the closed graph theorem, we obtain $\iota$ is a bounded operator. That is,
\[
\| \uu \|_{\mathcal Y} \lesssim \|\uu\|_{\mathcal X}
\]
implying that
\begin{equation*}
\| \uu(0)\|_{H^1(\Og_0)} + \|\uu'(0)\|_{L^2(\Og_0)} \lesssim
\|\uu(\pm T)\|_{H^1(\Omega^c)} + \|\uu'(\pm T)||_{L^2(\Omega^c)} + \|\uu\|_{L^2((-T,T)\times \R^n)}.
\end{equation*}
Since $\uu$ is even in $t$ and $\uu'(0) = 0$, we obtain the observability estimate in the statement of the proposition.

Let us now proceed to prove the claim. To this end, we will analyze the $H^1$ regularity of $\uu \in \mathcal{X}$. Let $(\x,\xi) \in S^*\Omega$, from Assumption~\ref{A:sharp}, there exists $\eps>0$ such that either $\gamma^{p/s}_{\x,\xi}(T) \in  \Og'_\eps$ or $\gamma^{p/s}_{\x,\xi}(-T) \in  \Og'_\eps$, where $\Og'_\eps :=\{\x \in \R^3: \dist(\x,\Og) >\eps\}$\footnote{Here the distance is in the Euclidean metric.}. From the finite speed of propagation (see Corollary~\ref{C:Domain}) and the fact that $\U(\pm T) \in H^1(\Og^c) \times L^2(\Og^c)$, we obtain $\uu \in H^1((T-(\eps/c_+),T+(\eps/c_+)) \times \Og_\eg)$ and $\uu \in H^1((-T-(\eps/c_+),- T+(\eps/c_+)) \times \Og_\eg)$.  Therefore, $\uu$ is microlocally in $H^1$ on the bicharacteristic $(t,\tau = \pm |\xi|_{p/s},\gamma^{p/s}_{\x,\xi}(t),\dot{\gamma}^{p/s}_{\x,\xi}(t))$ near either $t= T$ or $t=-T$. Propagating both the $p$ and $s$ bicharactersitics implies $\uu$ is microlocally in $H^1$ on the whole aforementioned bicharacteristic (see, e.g., \cite[Corollary 1.3]{taylor1978propagation} and also \cite{dencker1982propagation,hansen2003propagation,rachele2000inverse}). In particular, $\uu$ is in $H^1$ near $(0,\tau = \pm |\xi|_{p/s},\x,\xi)$. Since this holds for each $\xi$, $\uu$ is microlocally in $H^1$ near the portion of $\Sigma_{\Box^*}$ around $t=0$ and $\x\in \Omega$.  Here, $\Sigma_{\Box^*}$ is the set of characteristic points for the elastic wave operator $\Box^*$: $$\Sigma_{\Box^*} = \{(t,\x,\tau,\xi) \in T^*(\R \times \Omega)\setminus 0; \tau^2 = |\xi|^2_{p/s}\}.$$
Outside of $\Sigma_{\Box^*}$, $\Box^*$ is elliptic. Therefore, $\uu$ belongs to $H^1$ microlocally at such points, by microlocal elliptic regularity. Thus, $\uu \in H^1$ locally near $\{0\} \times \Omega$. By energy estimates for the elastic wave equation, this implies that the Cauchy data $\U(0)|_{\Og_0}$ is in the desired space. This proves the claim.

\end{proof}

\begin{lemma} \label{L:Injective}
The mapping $\f \to \big(\uu(T), \uu_t(T)\big)|_{\Og^c}$ is injective.
\end{lemma}

\begin{proof}
Assuming that $\big(\uu(T), \uu_t(T)\big)|_{\Og^c} \equiv 0$, let us prove $\f \equiv 0$. Let us extend $\uu$ as an even function in the time variable $t$. Then, $\uu$ satisfies the elastic wave equation for all time $t \in \R$.  Our main idea is to prove the claim: $\uu=0$ on $[-T,T] \times \Og^c$. Then applying the unique continuation principle (Theorem~\ref{T:Unique}) we obtain $\uu(t,\x) =0$ for all $(t,\x)$ such that $\dist_s(\x,\mS) + |t| \leq T$. This implies $\f = \uu(0) = 0$ on $\Og$.\footnote{Assumption~\ref{A:sharp} implies $\dist_s(\x,\mS) <T$.}

Let us now proceed to prove the claim, for which we make use of the domain of dependence principle (Theorem~\ref{T:Domain}), unique continuation principle (Theorem~\ref{T:Unique}), and a simple geometric (distance peeling) argument. From the domain of dependence principle (Theorem~\ref{T:Domain}), we obtain
$\uu(t,\x) =0$ for all $$K_+ = \{(t,\x): \x \in \Og^c, \dist(\x,\mS:=\pdh \Og) \geq c_+ |t-T|\}.$$
Using the same argument as above for negative time, we obtain
$u(x,t) =0$ for all $$K_- = \{(t,\x): \x \in \Og^c, \dist(\x,\mS) \geq c_+ |t + T|\}.$$
We deduce that $\uu \equiv 0$ on $K = K_+ \cup K_-$. This in particular, implies $\uu=0$ for all $(t,\x) \in [- 2T, 2T] \times \Og^c$ such that $\dist(\x, \mS) \geq c_+ T$. Take $\x_0 \in \Og^c$ be such $\dist(\x_0,\Og) = c_+ T + \eps$. Since $\uu \equiv 0$ on a neighborhood of  $[-2T,2T] \times \{\x_0 \}$, applying the unique continuation principle, we have $\uu \equiv 0$ on
$$O_0 = \{(t,\x): \frac{\dist(\x,\x_0) }{c_-} + |t| \leq 2T\}. $$
That is, $\uu$ vanishes on union $K_+ \cup K_- \cup O_0$. For the sake of visualization, we consider the one dimensional picture (the realistic three dimensional scenario follows in the same manner). Then, the regions $K_+,K_-,O_0$ are visualized as in Fig~\ref{F:domain}~(a). Sending $\eps \to 0$ we obtain that $\uu$ vanishes on the union of $K_\pm$ and the grey triangle in Fig~\ref{F:domain}~(b). This union, in particular, contains the green line segment and its translation to the right (i.e., away from the domain $\Og$). Repeating the same argument, we obtain $\uu$ vanishes on the green and, the, orange triangles, see Fig~\ref{F:domain}~(c).

\begin{figure}[t]
\begin{center}
\begin{tabular}{ccc}
\includegraphics[width=1.5in]{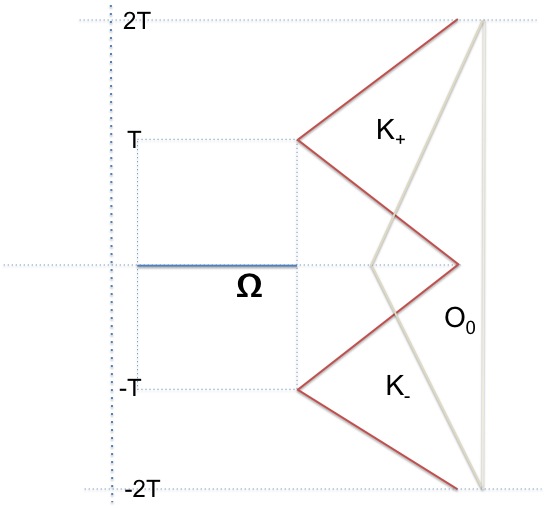} &
\includegraphics[width=1.5in]{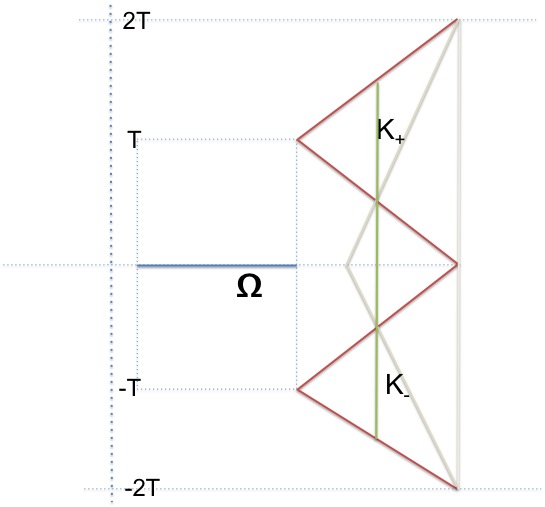} & 
\includegraphics[width=1.5in]{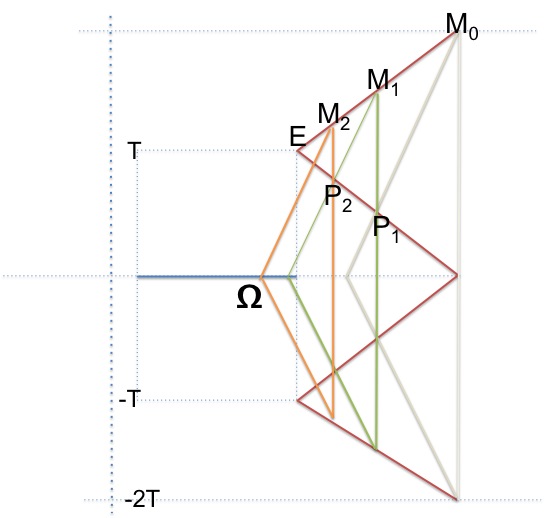} \\
(a)  & (b) & (c) \\
\end{tabular}
\end{center}
\caption{Distance peeling argument: (a) $\uu$ vanishes on the grey triangle due to the unique continuation principle (b) the limit when $\eps \to 0$ (c) repeat the unique continuation argument to move closer to the $\mS$.}
\label{F:domain}
\end{figure}

Continuing the process, we obtain a sequence of triangles on which $\uu$ vanishes. We only need to prove that the vertical edge of these triangles converges to the boundary. Let us first note that since the edges of these triangles are parallel, we obtain
$$\overline{E M_{2}}/\overline{E M_1} =  \overline{E P_{2}}/\overline{E P_1}= \overline{E M_{1}}/\overline{E M_0}.$$
In general, repeating the same argument, we obtain $\overline{E M_{i+1}}/\overline{E M_i}$ is independent of $i$, where $M_i$ is the highest vertex of the $(i+1)^{th}$ triangle. Therefore, $\lim_{i \to \infty} \overline{E M_{i}}=0$. This finishes our proof.
\end{proof}

\begin{proof}[\bf Proof of Theorem~\ref{T:Main}]
Let $U \subset \R^3$ and $\uu(t): U \to \R^3$, we denote
$$E_U(\uu,t)= \|\uu(t)\|_{H(U)}^2 + \|\uu_t(t)\|^2_{L^2(U)}.$$
Simple integration by parts shows that $E_{\R^3}(\uu,t)$ is independent of time. It is called the quadratic energy of the solution $\uu$.

Let us denote $\ww = \uu - \vv$, we obtain $\ww \in C(0,T; H_0^1(\Og)) \cap C^1(0,T; L^2(\Og))$. Simple integration by parts shows
$$E_{\Og}(\ww,0)= E_{\Og}(\ww,T) .$$
Noting that $\ww(T) = \mathcal{P}_\Og(\uu(T))$ and applying Lemma~\ref{L:proj}, we obtain
$$\|\ww(T)\|^2_{H(\Og)} \leq \|\uu(T)\|^2_{H(\Og)}.$$
Since $\ww_t(T) = \uu_t(T)$, we arrive at $E_{\Og}(\ww,T) \leq E_{\Og}(\uu,T)$ and hence 
\begin{eqnarray*}E_{\Og}(\ww,0) \leq E_{\Og}(\uu,T).\end{eqnarray*} 
Noting that $(\Id - A \Lambda) \f=\mathcal{P}_{\Og_0} (\ww(0))$ and applying Lemma~\ref{L:proj}, we obtain
$$\|(\Id - A \Lambda) \f\|^2_{H(\Og_0)} \leq E_{\Og_0}(\ww,0) \leq E_{\Og}(\ww,0),$$
which gives $$\|(\Id - A \Lambda) \f\|^2_{H(\Og_0)} \leq E_\Og(\uu,T).$$
We note that
$$E_\Og(\uu,T) = E_{\R^3} (\uu,T) - E_{\Og^c} (\uu,T) = E_{\R^3} (\uu,0) - E_{\Og^c} (\uu,T) = \|\f\|_{H(\Og_0)}^2 - E_{\Og^c} (\uu,T).$$
Therefore,
$$\|(\Id - A \Lambda) \f\|^2_{H(\Og_0)} \leq \|\f\|_{H(\Og_0)}^2 - E_{\Og^c} (\uu,T). $$
It now remains to prove that
$$\|\f\|^2_{H(\Og_0)} \lesssim \, E_{\Og^c}(\uu,T).$$
Recall from Lemma~\ref{L:lb}
$$\|\f\|_{H^1(\Og_0)} \lesssim  \|\uu(T)\|_{H^1(\Og^c)} +  \|\uu_t(T)\|_{L^2(\Og)}  + \|\uu\|_{L^2((0,T) \times \R^3)} .$$
Noting that $\|\f \|_{H^1(\Og_0)} \cong \|\f\|_{H(\Og_0)}$, we obtain
$$\|\f\|_{H(\Og_0)} \lesssim \|\uu(T)\|_{H(\Og^c)} + \|\uu_t(T)\|_{L^2(\Og)}  + \|\uu\|_{L^2((0,T) \times \R^3)}.$$
Let us note that the mapping $\f \to \uu$ is compact from $H_0(\Og_0)$ to $L^2((0,T) \times \R^3)$ and the mapping $\f \to (\uu(T),\uu_t(T))$ is injective from $H_0(\Og_0)$ to $H^1(\Og^c) \times L^2(\Og^c)$ (Lemma~\ref{L:Injective}).  Applying \cite[Theorem~V.3.1]{taylor1981pseudodifferential}, we obtain
$$\|\f\|_{H(\Og_0)} \lesssim  \|\uu(T)\|_{H^1(\Og^c)}+ \|\uu_t(T)\|_{L^2(\Og^c)} .$$
Due to the finite speed of propagation, $\uu(T)$ is supported inside the ball $B(0,R)$ for some fixed big enough $R$. Due to Korn's inequality (see \cite[Theorem 5]{kondrat1988boundary}),  $\|\uu(T)\|_{H^1(\Og'_R)})  \cong \|\uu(T)\|_{H(\Og'_R)}$ and hence $\|\uu(T)\|_{H^1(\Og^c)} \cong \|\uu(T)\|_{H(\Og^c)}$.
We obtain
$$\|\f\|^2_{H(\Og_0)} \lesssim  \|\uu(T)\|^2_{H(\Og^c)}+ \|\uu_t(T)\|^2_{L^2(\Og)}  =  E_{\Og^c}(\uu,T) .$$
This finishes our proof.
\end{proof}

\section*{Acknowledgment}
The first author thanks the Simons Foundation for financial support. The second author's research was partially supported by the NSF grants DMS 1212125 and DMS 1616904. He thanks Professor M. de Hoop and his group at Rice University for their hospitality.


\end{document}